\documentclass[12pt]{article}
\usepackage{amsmath,latexsym}
\usepackage{enumerate}
\usepackage{amsthm} 

\newcommand\RR{\rm {I \! R}}

\newcommand\NN{\rm {I \! N}}
\newcommand\FF{{\cal F}}
\newcommand\GG{{\cal G}}

\DeclareMathOperator\sP{P}   %probability, treated as an operator
\newcommand{\rP}{\mathrm{P}} %probability, when spacing isn't a worry
\DeclareMathOperator\sE{E}   
\newcommand{\rE}{\mathrm{E}} 
\DeclareMathOperator\sI{I}   
 
\newcommand{\eps}{\varepsilon}
\DeclareMathOperator\st{st}

\numberwithin{equation}{section}
\theoremstyle{plain}
\newtheorem{theorem}{Theorem}[section]
\newtheorem{definition}[theorem]{Definition}
\newtheorem{corollary}[theorem]{Corollary}
\newtheorem{example}[theorem]{Example}
\newtheorem{lemma}[theorem]{Lemma}
\newtheorem{remark}[theorem]{Remark}

\newtheorem{question}[theorem]{Question}
\numberwithin{equation}{section}

\newcommand\address{Address: Department of Mathematics, University of North Texas,
1155 Union Circle \#311430, Denton, TX 76203-5017, USA; E-mail: allaart@unt.edu}
\newcommand\thankyou{Supported in part by Japanese GCOE Program G08: ``Fostering Top Leaders in Mathematics --- Broadening the Core and Exploring New Ground".} 

\title{Predicting the supremum: optimality of ``stop at once or not at all"\footnote{\thankyou}}
\author{Pieter C. Allaart \footnote{\address}}
\date{\today}

\begin{document}

\maketitle

\begin{abstract}
Let $(X_t)_{0\leq t\leq T}$ be a one-dimensional stochastic process with independent and stationary increments, either in discrete or continuous time. This paper considers the problem of stopping the process $(X_t)$ ``as close as possible" to its eventual supremum $M_T:=\sup_{0\leq t\leq T}X_t$, when the reward for stopping at time $\tau\leq T$ is a nonincreasing convex function of $M_T-X_\tau$. Under fairly general conditions on the process $(X_t)$, it is shown that the optimal stopping time $\tau$ takes a trivial form: it is either optimal to stop at time $0$ or at time $T$. For the case of random walk, the rule $\tau\equiv T$ is optimal if the steps of the walk stochastically dominate their opposites, and the rule $\tau\equiv 0$ is optimal if the reverse relationship holds. An analogous result is proved for L\'evy processes with finite L\'evy measure. The result is then extended to some processes with nonfinite L\'evy measure, including stable processes, CGMY processes, and processes whose jump component is of finite variation.

\bigskip
{\it AMS 2000 subject classification}: 60G40, 60G50, 60J51 (primary); 60G25 (secondary)

\bigskip
{\it Key words and phrases}: Random walk; L\'evy process; optimal prediction; ultimate supremum; stopping time; skew symmetry; convex function
\end{abstract}

\section{Introduction}

In recent years there has been a great deal of interest in optimal prediction problems of the form
\begin{equation}
\sup_{\tau\leq T}\sE[f(M_T-X_\tau)],
\label{eq:intro-objective}
\end{equation}
where $f$ is a nonincreasing function, $(X_t)_{t\geq 0}$ a one-dimensional stochastic process, $T>0$ a finite time horizon, and $M_T:=\sup\{X_t: 0\leq t\leq T\}$. The supremum in \eqref{eq:intro-objective} is taken over the set of all stopping times adapted to the process $(X_t)_{t\geq 0}$ for which $\rP(\tau\leq T)=1$.
For the case of Brownian motion, the problem \eqref{eq:intro-objective} has been investigated for several reward functions $f$, though it is often formulated as a penalty-minimization problem in the form
\begin{equation}
\inf_{\tau\leq T}\sE[\tilde{f}(M_T-X_\tau)],
\label{eq:minimize-penalty}
\end{equation}
where $\tilde{f}:=-f$. For instance, Graversen et al.~\cite{GPS} solved \eqref{eq:minimize-penalty} for standard Brownian motion and $\tilde{f}(x)=x^2$. Their result was generalized to $\tilde{f}(x)=x^\alpha$ for arbitrary $\alpha>0$ by Pedersen \cite{Pedersen}, who also considered the function $f=\chi_{[0,\eps]}$ for $\eps>0$ in \eqref{eq:intro-objective}. Du Toit and Peskir \cite{DuToit1} were the first to extend these results (for power functions $f$) to Brownian motion with arbitrary drift, which required an entirely new approach. 
More recently, Shiryaev et al.~\cite{SXZ} considered the problem \eqref{eq:intro-objective} for Brownian motion with drift and $f(x)=e^{-\sigma x}$, where $\sigma>0$. In that case the problem has the natural interpretation of maximizing the expected ratio of the selling price to the eventual maximum price in the Black-Scholes model for stock price movements. They observed that when the drift parameter lies outside a certain critical interval, the optimal rule $\tau^*$ becomes trivial; that is, either $\tau^*\equiv 0$ or $\tau^*\equiv T$. A year later, Du Toit and Peskir \cite{DuToit2} managed to prove that the optimal rule is trivial also in the critical interval. More precisely, their result was that $\tau^*\equiv 0$ when the drift is negative, and $\tau^*\equiv T$ when the drift is positive. While this may seem intuitively quite plausible, it is nontrivial to prove. Since the optimal rule changes abrubtly from $0$ to $T$ as the drift parameter passes through $0$, Du Toit and Peskir \cite{DuToit2} called it a ``bang-bang" stopping rule. They also showed that for the (seemingly quite similar) problem \eqref{eq:minimize-penalty} with $\tilde{f}(x)=e^{\sigma x}$, the optimal rule is not of bang-bang form, but transitions from $\tau^*\equiv 0$ to $\tau^*\equiv T$ in a nontrivial way throughout the critical interval.

In the discrete-time setting, an analogous result for Bernoulli random walk was obtained later the same year by Yam et al.~\cite{YYZ}, using ideas from \cite{DuToit2}. Here we put $T=N$, a positive integer, and write $X_n$ instead of $X_t$, where $\{X_n\}_{0\leq n\leq N}$ is a simple random walk with parameter $p$. Yam et al.~\cite{YYZ} considered both the function $f=\chi_0$, the characteristic function of the set $\{0\}$ (in which case the expectation in \eqref{eq:intro-objective} is just the probability of stopping at the ``top" of the random walk) and the function $f(x)=e^{-\sigma x}$, and concluded that in both cases, the optimal rule is of bang-bang type. Precisely, the optimal rule is $\tau\equiv N$ when $p>1/2$; $\tau\equiv 0$ when $p<1/2$; or any stopping rule $\tau$ satisfying $\rP(X_\tau=M_\tau \ \mbox{or}\ \tau=N)=1$ when $p=1/2$. It is worth noting that the case $f=\chi_0$ had already been considered for general {\em symmetric} random walks more than 20 years earlier by Hlynka and Sheahan \cite{Hlynka}.

The results for both discrete and continuous time were recently extended in Allaart \cite{Allaart}, where it is shown that the bang-bang principle holds for both Bernoulli random walk and Brownian motion with drift whenever $f$ is nonincreasing and convex. Equivalently, it holds for problem \eqref{eq:minimize-penalty} when $\tilde{f}$ is nondecreasing and concave, which is the case, for instance, for the natural penalty function $\tilde{f}(x)=x^\alpha$ with $0<\alpha\leq 1$. Allaart \cite{Allaart} gives simple sufficient conditions on $f$ for the optimal rules to be unique in the discrete-time case, and necessary and sufficient conditions for the case of Brownian motion. 

The aim of the present paper is to extend the result further still, to include more general random walks as well as certain L\'evy processes.
%It is natural to ask whether the bang-bang principle holds for a larger class of processes, and if so, what characteristics a process must possess in order for it to satisfy the bang-bang principle. This paper aims to give at least a partial answer to this question. 
First, in Section \ref{sec:random-walk}, it is shown that the bang-bang principle holds for any random walk whose increments stochastically dominate their opposites, or vice versa (see Theorem \ref{thm:random-walk} below). In Section \ref{sec:levy} an analogous result is proved for L\'evy processes, first for the case of finite L\'evy measure (Theorem \ref{thm:Levy-finite}), then for the more general case (Theorem \ref{thm:Levy-general}). This appears to require some notion of drift, and therefore it seems necessary to impose some additional conditions pertaining to the ``small jumps" of the process. One of these conditions can be omitted in the case when $f$ is continuous and bounded (Theorem \ref{thm:bounded-f}), but the author does not know whether it is needed in the general case. The extra conditions may seem restrictive, but they are satisfied by several commonly studied types of L\'evy processes including subordinators, symmetric stable processes, and CGMY processes.

A possible application of this research is in finance. Suppose you buy a share of stock on the first day of the month, which you must sell some time by the end of the month. Perhaps the stock price follows a random walk in discrete time, and your objective is to maximize the probability of selling the stock at the highest price over the month. In that case, let $X_t$ be the random walk, and let $f=\chi_0$. Or perhaps the stock price follows an exponentiated L\'evy process, such as geometric Brownian motion, and your goal is to maximize the expected ratio of the price at the time you sell to the eventual maximum price. In that case, let $X_t$ be the L\'evy process, and put $f(x)=e^{-\sigma x}$, where $\sigma>0$. In both examples the results of this paper imply, under suitable conditions on the process $X_t$, that it is either optimal to sell the stock immediately, or to keep it until the last day of the month. In fact, the result for the second example remains valid if one takes as objective function an arbitrary increasing convex function $g$ of the price ratio, since if $g:(0,\infty)\to\RR$ is increasing and convex, then $f(x)=g(e^{-\sigma x})$ is decreasing and convex.

After this work was begun, the author learnt that D. Orlov has also extended the bang-bang principle to certain L\'evy processes. Unfortunately, an English version of his paper was not available at the time the present article was nearing completion. In addition, a paper by Bernyk et al. \cite{BDP2} appeared in which problem \eqref{eq:minimize-penalty} is solved for stable L\'evy processes of index $\alpha\in(1,2)$ with no negative jumps, for the penalty function $\tilde{f}(x)=x^p$ with $p>1$. (We observe that for this case, $f=-\tilde{f}$ is not convex, so the results of the present note do not apply; indeed, the optimal rule is nontrivial and its determination requires significant analytical tools.) Some of the preparatory work for this last paper was done in \cite{BDP1}. 

%Finally, it should be noted that for {\em non}-convex $f$, the problem \eqref{eq:intro-objective} is much more difficult to solve in general; see Du Toit and Peskir \cite[Section 3]{DuToit2} or Pedersen \cite{Pedersen}.

\section{The maximum of a random walk} \label{sec:random-walk}

In this section, let $\{X_n\}_{n=0,1,\dots}$ be a random walk with general steps satisfying a form of skew-symmetry as follows: $X_0\equiv 0$, and for $n\geq 1$, $X_n=\sum_{k=1}^n\xi_k$, where $\xi,\xi_1,\xi_2,\dots$ are independent, identically distributed (i.i.d.) random variables for which either $\xi\geq_{\st}-\xi$ or $\xi\leq_{\st}-\xi$. Here, $\geq_{\st}$ denotes the usual stochastic order of random variables, defined by
\begin{equation*}
X\geq_{\st} Y \quad\Longleftrightarrow\quad \sP(X>t)\geq \sP(Y>t) \quad\mbox{for all $t\in\RR$}.
\end{equation*}
(See Chapter 17 of Marshall and Olkin \cite{Marshall} for a general treatment of the stochastic order.) Let $M_n:=\max_{0\leq k\leq n}X_k$ for $n=0,1,\dots,N$, where $N\in\NN$ is a finite time horizon. For a nonincreasing function $f:[0,\infty)\to\RR$, consider the optimal stopping problem
\begin{equation}
\sup_{0\leq\tau\leq N}\rE[f(M_N-X_\tau)],
\label{eq:objective}
\end{equation}
where the supremum is over the set of all stopping times $\tau\leq N$ adapted to the natural filtration $\{\FF_n\}_{0\leq n\leq N}$ of the process $\{X_n\}_{0\leq n\leq N}$. We note that since $f$ is bounded above, the expectation in \eqref{eq:objective} always exists, though it could take the value $-\infty$.

The above setup includes Bernoulli random walk with arbitrary parameter $p\in(0,1)$ as a special case, but is of course much more general.

\begin{theorem} \label{thm:random-walk}
Assume that either $\xi\geq_{\st}-\xi$ or $\xi\leq_{\st}-\xi$, and let $f:[0,\infty)\to\RR$ be nonincreasing  and convex. Consider the problem \eqref{eq:objective}.%\vspace{-2mm}
\begin{enumerate}[(i)]%\setlength{\itemsep}{-1mm}
\item If $\xi\geq_{\st}-\xi$, the rule $\tau\equiv N$ is optimal.
\item If $\xi\leq_{\st}-\xi$, the rule $\tau\equiv 0$ is optimal.
\item If $\xi\stackrel{d}{=}-\xi$, any rule $\tau$ satisfying
%\begin{equation*}
$\rP(X_\tau=M_\tau\ \mbox{or}\ \tau=N)=1$
%\end{equation*}
is optimal.
\end{enumerate}
\end{theorem}

\begin{remark}
{\rm
By the assumption of convexity $f$ must be continuous on $(0,\infty)$, but it may have a jump discontinuity at $x=0$. Thus, in particular, the theorem covers the important case $f=\chi_0$, the characteristic function of the set $\{0\}$. In that case, the problem comes down to maximizing the probability of stopping at the highest point of the walk, so it can be thought of as a random walk version of the secretary (or best-choice) problem.
}
\end{remark}

\begin{remark}
{\rm
The condition $\xi\geq_{\st}-\xi$ holds for any random variable $\xi$ whose distribution is symmetric about some point $m\geq 0$, as is easy to see. For instance, any normal random variable $\xi$ with a nonnegative mean satisfies $\xi\geq_{\st}-\xi$. It follows that Theorem \ref{thm:random-walk} applies to all $\xi$ with symmetric distributions.
}
\end{remark}

\begin{example}
{\rm
An example of a nonsymmetric distribution for which $\xi\geq_{\st}-\xi$ is the Gumbel extreme value distribution, with distribution function $F(x)=\exp(-e^{-x})$, $x\in\RR$. To see this, let $g(x)=\exp(-e^x)+\exp(-e^{-x})$. Then
\begin{equation*}
g'(x)=\exp(-e^x+x)\left[\exp(e^x-e^{-x}-2x)-1\right].
\end{equation*}
Since it is easy to see (for instance by using a series expansion) that $e^x-e^{-x}-2x\geq 0$ for $x\geq 0$, it follows that $g$ is increasing on $[0,\infty)$. And since $\lim_{x\to\infty}g(x)=1$, this means that $g(x)<1$ for $x\geq 0$. Hence,
\begin{equation*}
1-F(x)\geq F(-x), \qquad x\geq 0.
\end{equation*}
So if $\xi\sim F$, then $\xi\geq_{\st}-\xi$.
}
\end{example}

\begin{example}
{\rm
The condition $\xi\geq_{\st}-\xi$ in statement (i) cannot be replaced by the condition $\rE(\xi)\geq 0$. For instance, let $\rP(\xi=3)=1/3=1-\sP(\xi=-1)$, let $f=\chi_0$, and take $n=2$. Even though $\rE(\xi)=1/3>0$, the optimal rule is easily seen to be $\tau\equiv 0$ rather than $\tau\equiv 2$.
}
\end{example}

In case of Bernoulli random walk, simple sufficient conditions on the function $f$ such that the optimal rules given above be unique are given in \cite{Allaart}. There an example is also given to show that without convexity of $f$, the conclusion of Theorem \ref{thm:random-walk} may fail in general.

The proof of Theorem \ref{thm:random-walk} uses the following generalization of Lemma 2.1 in \cite{Allaart}. Note that, compared to that lemma, a somewhat different method of proof is needed here.

\begin{lemma} \label{lem:key-inequality}
Let $f$ be as in Theorem \ref{thm:random-walk}, and suppose $\xi\geq_{\st}-\xi$. Then
\begin{equation}
\sE[f(z\vee M_n-X_n)]\geq \sE\big[f\big(z\vee (M_n-X_n)\big)\big]
\label{eq:key-inequality}
\end{equation}
for all $n\leq N$ and all $z\geq 0$.
\end{lemma}

Since the statement of the lemma involves only expectations, we may construct the random walk on a convenient probability space. Recall first that if $X\geq_{\st}Y$, then $X$ and $Y$ can be defined on a common probability space $(\Omega,\FF,\sP)$ so that $X(\omega)\geq Y(\omega)$ for all $\omega\in\Omega$. (See, for instance, \cite[Theorem 17.B.1]{Marshall}.) Thus, on a sufficiently large probability space, we can construct the random variables $\xi_1,\dots,\xi_N$ together with another set of random variables  $\tilde{\xi}_1,\dots,\tilde{\xi}_N$ such that the random vectors $(\xi_1,\tilde{\xi}_1),\dots,(\xi_N,\tilde{\xi}_N)$ are independent, $\tilde{\xi}_i\stackrel{d}{=}-\xi_1$ for each $i$, and $\xi_i\geq\tilde{\xi}_i$ for each $i$. Let
$\tilde{X}_0\equiv 0$, and $\tilde{X}_n=\sum_{k=1}^n\tilde{\xi}_k$, for $n=1,2,\dots,N$. Finally, define $\tilde{M}_n:=\max_{0\leq k\leq n}\tilde{X}_k$, $n=0,1,\dots,N$. Clearly,
$X_n\geq\tilde{X}_n$ and $M_n\geq\tilde{M}_n$ for every $n$.

It is also useful to define
\begin{equation*}
Z_n:=M_n-X_n \quad\mbox{and}\quad \tilde{Z}_n:=\tilde{M}_n-\tilde{X}_n, \qquad n=0,1,\dots,N.
\end{equation*}
One checks easily that
\begin{equation}
Z_n\leq\tilde{Z}_n, \qquad n=0,1,\dots,N.
\label{eq:Z-relationship}
\end{equation}

The key to the proof of the lemma is that, for each fixed $n$,
\begin{equation}
(M_n-X_n,X_n)\stackrel{d}{=}(\tilde{M}_n,-\tilde{X}_n),
\label{eq:reflection}
\end{equation}
as follows from an easy time-reversal argument.

\begin{proof}[Proof of Lemma \ref{lem:key-inequality}] 
The lemma holds trivially (with equality) when $z=0$, so assume $z>0$. We must first deal separately with the case when $\sE[f(z\vee M_n-X_n)]=-\infty$. Let $\alpha:=[f(z)-f(0)]/z$. Then the convexity of $f$ implies that, for all $u\geq 0$,
\begin{equation}
f(u+z)-f(u)\geq\alpha z.
\label{eq:convexity-bound}
\end{equation}
Using the algebraic inequality $z\vee m-x\leq z\vee(m-x)+z$ (valid for $z\geq 0$ and $m\geq 0$) and the fact that $f$ is nonincreasing, we get
\begin{equation*}
f(z\vee m-x)\geq f(z\vee(m-x)+z)\geq f(z\vee(m-x))+\alpha z,
\end{equation*}
in view of \eqref{eq:convexity-bound}. Thus, if $\sE[f(z\vee M_n-X_n)]=-\infty$, then $\sE\big[f\big(z\vee (M_n-X_n)\big)\big]=-\infty$ as well, and the lemma holds in this case.

Assume for the remainder of the proof that $\sE[f(z\vee M_n-X_n)]>-\infty$.
Since $n$ is fixed, we omit the subscripts and write $M=M_n$, $X=X_n$, $Z=Z_n$, and similarly for their tilded counterparts. Let
\begin{equation*}
h(z,m,x):=f(z\vee m-x)-f\big(z\vee(m-x)\big),
\end{equation*}
so that it is to be shown that
\begin{equation}
\rE[h(z,M,X)]\geq 0.
\label{eq:h-expectation}
\end{equation}
The above expectation exists and is finite, because $\alpha z\leq h(z,m,x)\leq|\alpha|z$. We begin by writing
\begin{equation*}
\rE[h(z,M,X)]=\sE[h(z,M,X);X>0]+\sE[h(z,M,X);X<0].
\end{equation*}
Using \eqref{eq:reflection}, we can write the second expectation as
\begin{equation}
\sE[h(z,M,X);X<0]=\sE[h(z,\tilde{M}-\tilde{X},-\tilde{X});\tilde{X}>0].
\label{eq:rewrite-second}
\end{equation}
On the other hand, we claim that
\begin{equation}
\rE[h(z,M,X);X>0]\geq\sE[h(z,\tilde{M},\tilde{X});\tilde{X}>0].
\label{eq:expectation-relationship}
\end{equation}
To see this, note that $h(z,M,X)=0$ on $\{X>0, M-X>z\}$, and hence,
\begin{align*}
h(z,M,X)\sI(X>0)&=\big(f(z\vee M-X)-f(z)\big)\sI(X>0, M-X\leq z)\\
&=\big(f(\max\{z-X,Z\})-f(z)\big)\sI(X>0, Z\leq z)\\
&\geq \big(f(\max\{z-X,Z\})-f(z)\big)\sI(\tilde{X}>0, \tilde{Z}\leq z)\\
&\geq \big(f(\max\{z-\tilde{X},\tilde{Z}\})-f(z)\big)\sI(\tilde{X}>0, \tilde{Z}\leq z)\\
&=h(z,\tilde{M},\tilde{X})\sI(\tilde{X}>0).
\end{align*}
Here the first inequality follows since $\{\tilde{X}>0, \tilde{Z}\leq z\}\subset \{X>0, Z\leq z\}$ by \eqref{eq:Z-relationship}, $\max\{z-X,Z\}\leq z$ on $\{X>0, Z\leq z\}$, and $f$ is nonincreasing. The second inequality follows since $f$ is nonincreasing and $\max\{z-X,Z\}\leq \max\{z-\tilde{X},\tilde{Z}\}$.

Combining \eqref{eq:rewrite-second} and \eqref{eq:expectation-relationship}, we obtain
\begin{equation}
\rE[h(z,M,X)]\geq \sE[h(z,\tilde{M},\tilde{X})+h(z,\tilde{M}-\tilde{X},-\tilde{X});\tilde{X}>0].
\label{eq:together}
\end{equation}

Next, the convexity of $f$ implies that for all $0\leq x<y$ and all $d>0$,
\begin{equation}
f(x)-f(x+d)\geq f(y)-f(y+d),
\label{eq:convex-property}
\end{equation}
as is easily checked. Thus, for $z\geq 0$ and $0<x\leq m$, we have
\begin{align*}
h(z,m,x)+h(z,m-x,-x)&=\big[f(z\vee m-x)-f\big(z\vee(m-x)\big)\big]\\
&\qquad+\big[f\big(z\vee(m-x)+x\big)-f(z\vee m)\big]\\
&=[f(z\vee m-x)-f(z\vee m)]\\
&\qquad-\big[f\big(z\vee(m-x)\big)-f\big(z\vee(m-x)+x\big)\big]\\
&\geq 0,
\end{align*}
where the inequality follows by \eqref{eq:convex-property} with $d=x$, since $x>0$ implies that $z\vee m-x\leq z\vee(m-x)$. This, together with \eqref{eq:together}, yields \eqref{eq:h-expectation}. 
\end{proof}

\begin{corollary} \label{cor:consequence}
Under the hypotheses of Lemma \ref{lem:key-inequality},
\begin{equation}
\sE[f(z\vee M_n-X_n)]\geq \sE[f(z\vee M_n)]
\label{eq:p-inequality}
\end{equation}
for all $n\leq N$ and all $z\geq 0$.
\end{corollary}

\begin{proof}
By \eqref{eq:reflection}, the inequality \eqref{eq:key-inequality} can be expressed alternatively as
\begin{equation}
\sE[f(z\vee M_n-X_n)]\geq \sE[f(z\vee \tilde{M}_n)].
\label{eq:alternative-form}
\end{equation}
But $\rE[f(z\vee \tilde{M}_n)]\geq \sE[f(z\vee M_n)]$, since $\tilde{M}_n\leq_{\st} M_n$ and $f$ is nonincreasing. Thus, \eqref{eq:p-inequality} follows. 
\end{proof}

\begin{proof}[Proof of Theorem \ref{thm:random-walk}] 
The main idea in the proof below is essentially due to Du Toit and Peskir \cite{DuToit2}; see Yam et al.~\cite{YYZ} for the discrete-time case. 

(i) Suppose first that $\xi_1\geq_{\st}-\xi_1$. Construct the random variables $\xi_k$, $X_k$, $M_k$, $Z_k$ and $\tilde{\xi}_k$, $\tilde{X}_k$, $\tilde{M}_k$ and $\tilde{Z}_k$ on a common probability space as in the discussion following the statement of Lemma \ref{lem:key-inequality}. 
Define the $\sigma$-algebras
\begin{equation*}
\GG_k:=\sigma(\{\xi_1,\dots,\xi_k,\tilde{\xi}_1,\dots,\tilde{\xi}_k\}), \qquad k=0,1,\dots,N.
\end{equation*}
It will be important later in the proof that the increments $X_k-X_j$ and $\tilde{X}_k-\tilde{X}_j$ are independent of $\GG_j$, for all $0\leq j\leq k$. Note further that if the stopping time $\tau\equiv N$ is optimal among the set of all stopping times relative to the filtration $\{\GG_k\}$, then it is certainly optimal among the stopping times relative to $\{\FF_k\}$. Thus, it is sufficient to show that
\begin{equation}
\rE[f(M_N-X_\tau)]\leq\sE[f(M_N-X_N)]
\label{eq:domination}
\end{equation}
for any stopping time $\tau$ relative to $\{\GG_k\}$. Define the functions
\begin{equation*}
G(k,z):=\sE[f(z\vee M_k)],\qquad D(k,z):=\sE[f(z\vee M_k-X_k)],
\end{equation*}
for $z\geq 0$ and $k=0,1,\dots,N$. Note that $G(k,z)$ and $D(k,z)$ can possibly take the value $-\infty$. Let $\tau\leq N$ be any stopping time. An easy exercise using the independent and stationary increments of the random walk $\{X_k\}$ leads to
\begin{equation}
\rE[f(M_N-X_\tau)|\GG_\tau]=G(N-\tau,Z_\tau),
\label{eq:first-conditional}
\end{equation}
and
\begin{equation}
\rE[f(M_N-X_N)|\GG_\tau]=D(N-\tau,Z_\tau).
\label{eq:second-conditional}
\end{equation}
Now Corollary \ref{cor:consequence} says that $D(k,z)\geq G(k,z)$, and hence
\begin{equation*}
\rE[f(M_N-X_\tau)|\GG_\tau]\leq \sE[f(M_N-X_N)|\GG_\tau].
\end{equation*}
Taking expectations on both sides gives \eqref{eq:domination}, as desired.

(ii) Suppose next that $\xi_1\leq_{\st}-\xi_1$. Apply again the construction following the statement of Lemma \ref{lem:key-inequality}, but this time with $\xi_i\leq\tilde{\xi}_i$ for all $i$. Observe that all the other relationships between random variables and their tilded counterparts are now reversed as well, i.e.
\begin{equation*}
X_k\leq \tilde{X}_k, \qquad M_k\leq\tilde{M}_k, \qquad Z_k\geq\tilde{Z}_k,
\end{equation*}
for $k=0,1,\dots,N$. Define the filtration $\{\GG_k\}$ and the function $G(k,z)$ as in the proof of part (i) above, and let
\begin{equation*}
\tilde{D}(k,z):=\sE[f(z\vee \tilde{M}_k-\tilde{X}_k)].
\end{equation*}
In place of \eqref{eq:alternative-form}, we now have the inequality
\begin{equation*}
\sE[f(z\vee \tilde{M}_k-\tilde{X}_k)]\geq \sE[f(z\vee M_k)],
\end{equation*}
or in other words, $\tilde{D}(k,z)\geq G(k,z)$. Furthermore, the fact that $f$ is nonincreasing implies that $G(k,z)$ is nonincreasing in $z$, and therefore,
\begin{equation*}
G(N-j,Z_j)\leq G(N-j,\tilde{Z}_j)
\end{equation*}
for each $j$. By \eqref{eq:reflection}, $\rE[f(M_N)]=\sE[f(\tilde{Z}_N)]$. 
Putting these facts together, we obtain for any stopping time $\tau$ relative to $\{\GG_k\}$, by the same kind of reasoning as in the proof of part (i),
\begin{align}
\begin{split}
\rE[f(M_N-X_\tau)]&=\sE[G(N-\tau,Z_\tau)]\leq \sE[G(N-\tau,\tilde{Z}_\tau)]\\
&\leq \sE[\tilde{D}(N-\tau,\tilde{Z}_\tau)]=\sE[f(\tilde{Z}_N)]=\sE[f(M_N)].
\end{split}
\label{eq:chain}
\end{align}
Hence, the rule $\tau\equiv 0$ is optimal.

(iii) Suppose finally that $\xi_1\stackrel{d}{=}-\xi_1$. This is a special case of part (i), so the rule $\tau\equiv N$ is optimal.
% Here, however, $\tilde{Z}_k=Z_k$ for all $k$, so that the first inequality in \eqref{eq:chain} holds with equality for any stopping time $\tau$. Now
Now let $\tau$ be any stopping time such that with probability one, $X_\tau=M_\tau$ or $\tau=N$. Since $G(0,z)=f(z)=D(0,z)$ for all $z\geq 0$ and $G(k,0)=\sE[f(M_k)]=\sE[f(\tilde{Z}_k)]=\sE[f(Z_k)]=D(k,0)$ for all $k$, \eqref{eq:first-conditional} and \eqref{eq:second-conditional} give equality in \eqref{eq:domination}. Hence, $\tau$ is optimal. 
\end{proof}

\section{The maximum of a L\'evy process} \label{sec:levy}

A careful study of the proofs in the previous section reveals that the essential property of the random walk is its independent and stationary increments. Furthermore, in order to construct the random walk $\{X_n\}$ and its dual  $\{\tilde{X}_n\}$ on a common probability space in such a way that the increments of $\{X_n\}$ uniformly dominate those of $\{\tilde{X}_n\}$ (or vice versa), the step-size distribution had to satisfy a type of skew symmetry. With this in mind, we can now extend the result to a much larger class of stochastic processes. 

The general continuous-time analog of a random walk is a {\em L\'evy process}, which is defined as a stochastic process on $[0,\infty)$ with independent and stationary increments which starts at $0$ and is continuous in probability. Following standard practice, we assume also that the process has almost surely right-continuous sample paths with left-hand limits everywhere (or, for short, that the process is {\em rcll}). If $X=(X_t)_{t\geq 0}$ is a (one-dimensional) L\'evy process, it is uniquely determined by the L\'evy-Khintchine formula
\begin{equation*}
\rE\left[e^{iuX_t}\right]=e^{t\eta(u)},
%\label{eq:LK-representation-1}
\end{equation*}
where
\begin{equation}
\eta(u)=i\gamma u-\frac{\sigma^2 u^2}{2} +\int_{\RR\backslash\{0\}}\left[e^{iuy}-1-iuy\chi_{(-1,1)}(y)\right]\nu(dy).
\label{eq:LK-representation-2}
\end{equation}
In this expression, the {\em L\'evy measure} $\nu$ satisfies $\int_{\RR\backslash\{0\}}(y^2\wedge 1)\nu(dy)<\infty$, but $\nu$ need not be finite. We say that $X$ is {\em generated by the triplet $(\gamma,\sigma^2,\nu)$}.

%In order to ensure that all expectations in this section are finite, we assume in addition that
%\begin{equation}
%\int_{|y|\geq 1}|y|\nu(dy)<\infty.
%\label{eq:finite-mean}
%\end{equation}
%It is well known (e.g. \cite[Theorem 2.5.2]{Applebaum}) that this condition guarantees the integrability of $X_t$ for all $t>0$.

Define the supremum process $M=(M_t)_{t\geq 0}$ by
\begin{equation*}
M_t:=\sup_{0\leq s\leq t}X_s, \qquad t\geq 0.
\end{equation*}
%As a consequence of \eqref{eq:finite-mean}, we have (e.g. \cite[Theorem 25.18]{Sato}) that 
%\begin{equation}
%\rE(M_t)<\infty \qquad\mbox{for all $t>0$}.
%\label{eq:integrable-sup}
%\end{equation}

If $\nu$ is finite, then $X$ is simply the sum of a Brownian motion with drift and a compound Poisson process, and it is straightforward to adapt the result of the previous section. This is done in Subsection \ref{subsec:interlacing} below. If $\nu$ is not finite, however, complications arise in attempting to couple the process $X$ with its dual, and some additional conditions appear to be needed to overcome these difficulties. This is made precise in Subsection \ref{subsec:general}. Finally, in Subsection \ref{subsec:bounded}, we eliminate one of the extra conditions in the case when $f$ is continuous and bounded.

\subsection{The case of finite $\nu$} \label{subsec:interlacing}

We consider first the case when $\nu$ is finite. Then we may put
\begin{equation*}
b:=\gamma-\int_{0<|y|<1}y\nu(dy),
\end{equation*}
%and write the L\'evy-Khintchine formula as
%\begin{equation}
%\eta(u)=ibu-\frac{\sigma^2 u^2}{2} +\int_{\RR\backslash\{0\}}\left(e^{iuy}-1\right)\nu(dy).
%\label{eq:finite-LK-formula}
%\end{equation}
and express $X_t$ pathwise in the form
\begin{equation}
X_t=bt+\sigma B_t+\sum_{i=1}^{N(t)}\xi_i,
\label{eq:path-representation}
\end{equation}
where $B_t$ is a standard Brownian motion, $\xi_1, \xi_2,\dots$ are i.i.d. random variables with distribution $\nu/|\nu|$, and $(N(t))_{t\geq 0}$ is a Poisson process with intensity $|\nu|$. In this representation, the Poisson process, the Brownian motion and the $\xi_i$'s are all independent of one another.

\begin{definition}
Let $X=(X_t)_{t\geq 0}$ be a L\'evy process of the form \eqref{eq:path-representation}, with finite L\'evy measure $\nu$.%\vspace{-1mm}
\begin{enumerate}[(i)]%\setlength{\itemsep}{-1mm}
\item $X$ is {\em right skew symmetric (RSS)} if $b\geq 0$ and  $\nu\big((a,\infty)\big)\geq\nu\big((-\infty,-a)\big)$ for all $a>0$.
\item $X$ is {\em left skew symmetric (LSS)} if $b\leq 0$ and  $\nu\big((a,\infty)\big)\leq\nu\big((-\infty,-a)\big)$ for all $a>0$.
\item $X$ is {\em symmetric} if $b=0$ and $\nu\big((a,\infty)\big)=\nu\big((-\infty,-a)\big)$ for all $a>0$.
\end{enumerate}
\end{definition}

Note that the condition regarding $\nu$ in the definition of RSS is equivalent to $\xi_1\geq_{\st}-\xi_1$, because if the inequality holds for all $a>0$, it holds for all $a\in\RR$.
The following result is the analog of Theorem \ref{thm:random-walk} for a L\'evy process with finite L\'evy measure $\nu$.

\begin{theorem} \label{thm:Levy-finite}
Let $X=(X_t)_{t\geq 0}$ be a L\'evy process with finite L\'evy measure $\nu$, adapted to a filtration $\{\FF_t\}$, such that $X_t-X_s$ is independent of $\FF_s$ for all $0\leq s\leq t$. Assume $X$ is either RSS or LSS, and let $f$ be as in Theorem \ref{thm:random-walk}. For fixed $T>0$, consider the problem
\begin{equation}
\sup_{0\leq\tau\leq T} \sE[f(M_T-X_\tau)],
\label{eq:Levy-objective}
\end{equation}
where the supremum is over all stopping times $\tau$ relative to the filtration $\{\FF_t\}$ with $\rP(\tau\leq T)=1$.%\vspace{-1mm}
\begin{enumerate}[(i)]%\setlength{\itemsep}{-1mm}
\item If $X$ is RSS, the rule $\tau\equiv T$ is optimal.
\item If $X$ is LSS, the rule $\tau\equiv 0$ is optimal.
\item If $X$ is symmetric, any rule $\tau$ satisfying $\rP(X_\tau=M_\tau\ \mbox{or}\ \tau=T)=1$ is optimal.
\end{enumerate}
\end{theorem}

If $\nu=0$, then $X$ is a Brownian motion with drift. Thus, the above theorem generalizes recent results of Shiryaev et al.~\cite{SXZ}, Du Toit and Peskir \cite[Section 4]{DuToit2} and Allaart \cite{Allaart}.

\begin{definition}
Let $X=(X_t)_{t\geq 0}$ be a L\'evy process . The {\em dual process} of $X$, denoted $\tilde{X}$, is a process such that $(\tilde{X}_t)_{t\geq 0}\stackrel{d}{=}(-X_t)_{t\geq 0}$. The {\em dual supremum process}, denoted $\tilde{M}$, is the process defined by $\tilde{M}_t:=\sup_{0\leq s\leq t}\tilde{X}_s$, for $t\geq 0$.
\end{definition}

If $X$ is a L\'evy process generated by the triplet $(\gamma,\sigma^2,\nu)$, then $\tilde{X}$ is a L\'evy process with triplet $(-\gamma,\sigma^2,\tilde{\nu})$, where $\tilde{\nu}(A)=\nu(-A)$ for any Borel set $A\subset \RR$. Note that if $X$ is RSS, then $\tilde{X}$ is LSS and vice versa.

\begin{lemma} \label{lem:Levy-reflection}
Let $X$ be any L\'evy process. Then, for each fixed $t\geq 0$,
\begin{equation*}
(M_t-X_t,X_t)\stackrel{d}{=}(\tilde{M}_t,-\tilde{X}_t).
%\label{eq:Levy-reflection}
\end{equation*}
\end{lemma}

\begin{proof}
This is essentially a know fact. Let $I_t:=\inf\{X_s: 0\leq s\leq t\}$. Then plainly
\begin{equation*}
(\tilde{M}_t,-\tilde{X}_t) \stackrel{d}{=} (-I_t,X_t).
\end{equation*}
According to Proposition 3 of Bertoin \cite[p.~158]{Bertoin},
\begin{equation*}
(-I_t,X_t-I_t)\stackrel{d}{=}(M_t-X_t,M_t),
\end{equation*}
which is equivalent to
\begin{equation*}
(-I_t,X_t)\stackrel{d}{=}(M_t-X_t,X_t).
\end{equation*}
Thus, the Lemma follows.
\end{proof}

\begin{proof}[Proof of Theorem \ref{thm:Levy-finite}] 
Assume for the moment that $X$ is RSS. Recall the representation \eqref{eq:path-representation}. On the same probability space on which the process $X$ is defined, we construct the dual $\tilde{X}$ as follows. For each $i\in\NN$, we can construct out of $\xi_i$ (using an external randomization if necessary) a random variable $\tilde{\xi}_i$ such that $\tilde{\xi}_i\stackrel{d}{=}-\xi_i$, and $\xi_i\geq\tilde{\xi}_i$ pointwise. Now put
\begin{equation*}
\tilde{X}_t:=-bt+\sigma B_t+\sum_{i=1}^{N(t)}\tilde{\xi}_i, \qquad t\geq 0.
\end{equation*}
Then it is easy to see that $(\tilde{X}_t)_{t\geq 0}\stackrel{d}{=}(-X_t)_{t\geq 0}$, and moreover, the processes $X$ and $\tilde{X}$ satisfy the property that, for all $0\leq s<t$ and for all $\omega\in\Omega$,
\begin{equation}
X_t(\omega)-X_s(\omega)\geq \tilde{X}_t(\omega)-\tilde{X}_s(\omega).
\label{eq:increment-domination}
\end{equation}
For $t\geq 0$, define
\begin{equation*}
Z_t:=M_t-X_t, \qquad \tilde{Z}_t:=\tilde{M}_t-\tilde{X}_t.
\end{equation*}
As in Section \ref{sec:random-walk}, it follows from \eqref{eq:increment-domination} that 
\begin{equation*}
M_t\geq\tilde{M}_t \qquad\mbox{and}\qquad Z_t\leq\tilde{Z}_t \qquad\mbox{for all $t\geq 0$}.
\end{equation*}
Using these relationships and Lemma \ref{lem:Levy-reflection}, we can show in exactly the same way as in the proof of Lemma \ref{lem:key-inequality}, that
\begin{equation*}
\sE[f(z\vee M_t-X_t)]\geq \sE\big[f\big(z\vee (M_t-X_t)\big)\big]
%\label{eq:Levy-key-inequality}
\end{equation*}
for all $t\geq 0$ and all $z\geq 0$.

Next, for $t\geq 0$, let $\GG_t$ be the smallest $\sigma$-algebra containing both $\FF_t$ and $\sigma(\{\tilde{X}_s:0\leq s\leq t\})$. Then $\{\GG_t\}_{t\geq 0}$ is a filtration with respect to which both $X$ and $\tilde{X}$ are adapted, and for each $0\leq s\leq t$, both $X_t-X_s$ and $\tilde{X}_t-\tilde{X}_s$ are independent of $\GG_s$. The rest of the proof is now the same (modulo subscript notation) as the proof of Theorem \ref{thm:random-walk}, where the analogs of \eqref{eq:first-conditional} and \eqref{eq:second-conditional} follow since $X$, being a L\'evy process, obeys the strong Markov property.
\end{proof}

\begin{question}
{\rm
It is clear that when $X$ is RSS, we have $X_t\geq_{\st}\tilde{X}_t$ for all $t\geq 0$. Does the converse of this statement hold?
}
\end{question}

\subsection{The general case} \label{subsec:general}

For a general L\'evy process with nonfinite L\'evy measure $\nu$, the construction of the previous subsection is no longer possible because the jump times are dense in the time interval $[0,T]$. Here we shall use the fact that a general L\'evy process on $[0,T]$ can always be obtained as the almost sure uniform limit of a sequence of processes of the form \eqref{eq:path-representation}. However, in order to ensure that this can be done while preserving the uniform domination of increments (i.e. \eqref{eq:increment-domination}), some extra conditions appear to be needed. Let the L\'evy-Khintchine representation of $X=(X(t))_{t\geq 0}$ be given by \eqref{eq:LK-representation-2}. (In what follows, it will be notationally more convenient to write $X(t)$ instead of $X_t$.)

\begin{definition}
We say $X$ is {\em balanced in its small jumps (BSJ)}, if
\begin{equation}
L:=\lim_{\eps\downarrow 0}\int_{\eps\leq |y|<1}y\nu(dy) \quad\mbox{exists and is finite}.
\label{eq:almost-symmetric}
\end{equation}
\end{definition}

This condition is always satisfied when $\nu$ is symmetric on a sufficiently small interval $(-\eps,\eps)$ where $\eps>0$, or when  $\int_{0<|y|<1}|y|\nu(dy)<\infty$. (In the latter case, the non-Gaussian part of $X$ has finite variation.) In the case when $\int_{0<|y|<1}|y|\nu(dy)=\infty$, \eqref{eq:almost-symmetric} may be interpreted as saying that $\nu$ is {\em almost} symmetric in a sufficiently small neighborhood of the origin. Roughly speaking, this means that we allow the small jumps of the process to be dense in time, provided that the positive and negative jumps more or less balance each other. It allows us to still think of the number $\gamma-L$ as the `drift' of the process.

It is clear that if $X$ is BSJ, then so is its dual $\tilde{X}$.

Denote by $\tilde{\nu}$ the dual measure of $\nu$, so that $\tilde{\nu}(A)=\nu(-A)$ for $A\subset\RR$. If $\mu$ and $\nu$ are measures on $\RR$ and $E\subset\RR$, we say $\mu$ {\em majorizes} $\nu$ on $E$ if $\mu(F)\geq\nu(F)$ for every $F\subset E$.

\begin{definition} \label{def:SRSS}
Let $X=(X(t))_{t\geq 0}$ be a L\'evy process.%\vspace{-2mm}
\begin{enumerate}[(i)]%\setlength{\itemsep}{-1mm}
\item We say $X$ is {\em strongly right skew symmetric (SRSS)} if all of the following hold:%\vspace{0mm}
\begin{enumerate}
\item $X$ is balanced in its small jumps;
\item $\gamma\geq L$, where $L$ is the limit in \eqref{eq:almost-symmetric};
\item $\nu\big((a,\infty)\big)\geq\nu\big((-\infty,-a)\big)$ for all $a>0$;
\item There exists $\eps>0$ such that $\nu$ majorizes $\tilde{\nu}$ on $(0,\eps)$.
\end{enumerate}
\item We say $X$ is {\em strongly left skew symmetric (SLSS)} if  $\tilde{X}$ is SRSS.
\item We say $X$ is {\em symmetric} if $\gamma=0$ and $\nu=\tilde{\nu}$.
\end{enumerate}
\end{definition}

\begin{remark}
{\rm
({\em a}) If $X$ is symmetric, then it is both SRSS and SLSS, since \eqref{eq:almost-symmetric} holds with $L=0$.

({\em b}) If $X$ is SRSS (resp. SLSS) and $\nu$ is finite, then $X$ is RSS (resp. LSS), since $b=\gamma-L$. The undesirable fourth condition in the definition of SRSS seems to be needed in order to carry out the pathwise construction of $X$ and its dual, below. At this point, the author does not see how to get around this technical difficulty, except in the special case when $f$ is bounded and continuous (see Subsection \ref{subsec:bounded} below).

({\em c}) The SRSS and SLSS conditions can be made more concrete in case $\nu$ has a density. Let $f,g: (0,\infty)\to[0,\infty)$ and suppose that
\begin{equation*}
\nu(dx)=\left(f(x)\chi_{(0,\infty)}(x)+g(-x)\chi_{(-\infty,0)}(x)\right)dx.
\end{equation*}
Then $\nu$ is a L\'evy measure if and only if $\int_0^\infty (x^2 \wedge 1)[f(x)+g(x)]\,dx<\infty$. The BSJ condition is now equivalent to convergence of the integral $\int_0^1 x[f(x)-g(x)]\,dx$. Conditions (c) and (d) in the definition of SRSS become, respectively

\bigskip
(c)'\ \ $\int_a^\infty [f(x)-g(x)]\,dx\geq 0$ for all $a>0$.

\medskip
(d)'\ \ There is $\eps>0$ such that $f(x)\geq g(x)$ for all $x\in(0,\eps)$.

\bigskip
The easiest way to satisfy both (c)' and (d)' is, of course, to take $f\geq g$ everywhere. This way, we may obtain nontrivial examples of nonsymmetric L\'evy processes that are SRSS (or SLSS). For instance, let
\begin{equation*}
f(x)=\frac{c}{x^p}, \qquad g(x)=\frac{c}{x^p+x^r}, \qquad\mbox{where} \quad c>0, \quad 2\leq p<3, \quad r>2p-2.
\end{equation*}
Then $r>p$, and $\nu$ satisfies \eqref{eq:almost-symmetric} with
\begin{equation*}
L=\int_0^1 x[f(x)-g(x)]\,dx=\int_0^1 \frac{cx^{r-2p+1}}{1+x^{r-p}}\,dx,
\end{equation*}
a convergent integral. Since (c)' and (d)' are obviously satisfied, the process will be SRSS if $\gamma\geq L$. Since $r>p$, the `large' positive jumps of $X(t)$ tend to be greater in magnitude (and occur more frequently) than the `large' negative jumps. On the other hand, the small jumps of the process in either direction are comparable in size.
}
\end{remark}

\begin{example} \label{ex:stable}
{\rm
({\em Stable processes}\,)
Let $X$ be a stable L\'evy process with index of stability $\alpha$ ($0<\alpha\leq 2$). If $\alpha=2$, then $X$ is just a Brownian motion with drift, and the optimal rule is already specified by Theorem \ref{thm:Levy-finite}. (In fact, in this case the optimal rules are unique except for some trivial cases; see Allaart \cite{Allaart}.)

If $\alpha<2$, then $\sigma=0$ and the L\'evy measure $\nu$ is of the form
\begin{equation*}
\nu(dx)=\left(\frac{c_1}{x^{1+\alpha}}\chi_{(0,\infty)}(x)+\frac{c_2}{|x|^{1+\alpha}}\chi_{(-\infty,0)}(x)\right)dx,
\end{equation*}
where $c_1\geq 0$, $c_2\geq 0$, and $c_1+c_2>0$ (see, e.g. Sato \cite{Sato}, p.~80). If follows that if $1\leq\alpha<2$, then $X$ is BSJ if and only if $c_1=c_2$, in which case $\nu$ is symmetric. In that case, $X$ is SRSS if $\gamma\geq 0$, and $X$ is SLSS if $\gamma\leq 0$. On the other hand, if $0<\alpha<1$, then the BSJ condition \eqref{eq:almost-symmetric} is always satisfied, with
\begin{equation*}
L=\int_{0<|x|<1}x\nu(dx)=\frac{c_1-c_2}{1-\alpha},
\end{equation*}
and $X$ is SRSS if $\gamma\geq L\geq 0$; or similarly, $X$ is SLSS if $\gamma\leq L\leq 0$.

Note that in the stable case, condition (d) in Definition \ref{def:SRSS} is satisfied whenever (a)-(c) are.
}
\end{example}

\begin{example} \label{ex:CGMY}
{\rm
({\em CGMY processes}\,)
Another example of nonsymmetric processes that are SRSS or SLSS is given by the CGMY processes, which are frequently used in financial modeling. The CGMY process, named for Carr, Geman, Madan and Yor (see \cite{CGMY}), is a L\'evy process with L\'evy measure
\begin{equation*}
\nu(dx)=C\cdot\frac{e^{-G|x|}\chi_{(-\infty,0)}(x) + e^{-Mx}\chi_{(0,\infty)}(x)}{|x|^{1+Y}},
\end{equation*}
where $C>0$, $G\geq 0$, $M\geq 0$ and $Y<2$, and it is assumed that $G>0$ and $M>0$ if $Y\leq 0$. The CGMY processes include the symmetric stable processes (take $G=M=0$) and are sometimes called {\em tempered stable processes}. The CGMY process with $Y=0$ is known as the {\em variance gamma process}. The very small jumps of a CGMY process behave essentially as in the symmetric stable case, and it is easy to check that all CGMY processes have the BSJ property. Furthermore, conditions (c) and (d) in Definition \ref{def:SRSS} are satisfied if and only if $M\leq G$. Hence, the CGMY process is SRSS if $M\leq G$ and $\gamma\geq L$, with $L$ as in \eqref{eq:almost-symmetric}; and it is SLSS if $M\geq G$ and $\gamma\leq L$.
}
\end{example}

%\begin{remark}
%{\rm
%If $\sigma=0$ and $\int_{|y|<1}|y|\nu(dy)<\infty$, then $X$ has finite variation, so there exist stopping times $\tau$ for which $\rP\big(X(\tau)=M(\tau)\big)>0$. Thus, the conclusion of Theorem \ref{thm:Levy-general} is meaningful even for $f=\chi_0$.
%}
%\end{remark}

We can now state the result for the most general case.

\begin{theorem} \label{thm:Levy-general}
Let $X=(X(t))_{t\geq 0}$ be a L\'evy process, and let $f$ be as in Theorem \ref{thm:random-walk}. For fixed $T>0$, consider the problem \eqref{eq:Levy-objective}.%\vspace{-2mm}
\begin{enumerate}[(i)]%\setlength{\itemsep}{-1mm}
\item If $X$ is SRSS, the rule $\tau\equiv T$ is optimal.
\item If $X$ is SLSS, the rule $\tau\equiv 0$ is optimal.
\item If $X$ is symmetric, any rule $\tau$ satisfying $\rP\big(X(\tau)=M(\tau)\ \mbox{or}\ \tau=T\big)=1$ is optimal.
\end{enumerate}
\end{theorem}

The proof of Theorem \ref{thm:Levy-general} hinges on the following construction. Once this is accomplished, the rest of the proof is the same as before.

\begin{lemma} \label{lem:general-construction}
Let $X$ be a SRSS L\'evy process. Then, on a suitable probability space $(\Omega,\FF,\sP)$, we can construct $X$ and its dual $\tilde{X}$ in such a way that there exists a set $\Omega_0\subset \Omega$ with $\rP(\Omega_0)=1$ such that, for all $0\leq s<t$ and for all $\omega\in\Omega_0$,
\begin{equation}
X(t;\omega)-X(s;\omega)\geq \tilde{X}(t;\omega)-\tilde{X}(s;\omega).
\label{eq:increment-domination-general}
\end{equation}
\end{lemma}

\begin{proof}
Let $\eps>0$ be as in the definition of SRSS. Then $X(t)$ can be expressed by the {\em L\'evy-Ito decomposition}
\begin{equation*}
X(t)=\gamma' t+\sigma B(t)+\int_{|y|<\eps} yN'(t,dy)+\int_{|y|\geq\eps}yN(t,dy),
%\label{eq:Levy-Ito}
\end{equation*}
where $B(t)$ is a standard Brownian motion on $\RR$, $\gamma':=\gamma-\int_{\eps\leq|y|<1}y\nu(dy)$, $(N(t,\cdot))_{t\geq 0}$ is a Poisson random measure with intensity measure $\nu$ which is independent of the Brownian motion, and $N'(t,\cdot)$ is defined by
\begin{equation*}
N'(t,dy)=N(t,dy)-t\nu(dy), \qquad t\geq 0.
\end{equation*}
In general, the integrals $\int_{|y|<\eps} yN(t,dy)$ and $\int_{|y|<\eps} y\nu(dy)$ need not converge, but the `compensated sum of small jumps', $\int_{|y|<\eps} yN'(t,dy)$, always does.

Now we will construct a sequence of L\'evy processes $Y_1,Y_2,\dots$ and their duals $\tilde{Y}_1,\tilde{Y}_2,\dots$, as follows. Let $\eps=\eps_1>\eps_2>\dots$ be a sequence of numbers decreasing to zero. Define first
\begin{equation*}
Y_1(t):=(\gamma-L)t+\sigma B(t)+\int_{|y|\geq \eps}yN(t,dy).
\end{equation*}
Then $Y_1$ has finite L\'evy measure $\nu_1$, where $\nu_1$ is the restriction of $\nu$ to the set $\{y:|y|\geq\eps\}$. Clearly $\nu_1\big((a,\infty)\big)\geq\nu_1\big((-\infty,-a)\big)$ for all $a>0$, since $\nu_1$ simply inherits this property from $\nu$. Since $\gamma\geq L$, we can construct $Y_1$ and its dual $\tilde{Y}_1$ on the same probability space so that these processes satisfy the increment property \eqref{eq:increment-domination}. Next, for $n\geq 2$, let
\begin{equation*}
Y_n(t)=\int_{\eps_n\leq|y|<\eps_{n-1}}yN(t,dy),
\end{equation*}
and note that by the usual independence property of Poisson point processes, the processes $Y_n$, $n\in\NN$ may be constructed independently of each other. Now for each $n\geq 2$, $Y_n$ is a compound Poisson process with (finite) L\'evy measure $\nu_n$, where $\nu_n$ is the restriction of $\nu$ to the set $\{y:\eps_n\leq|y|<\eps_{n-1}\}$. Since $\nu$ majorizes $\tilde{\nu}$ on $(0,\eps)$, it follows that $\nu_n\big((a,\infty)\big)\geq\tilde{\nu}_n\big((a,\infty)\big)$ for all $n\geq 2$. (Note that this fact would not be guaranteed without the fourth condition in the definition of SRSS.) Thus, we can construct $Y_n$ and its dual $\tilde{Y}_n$ together as in the previous subsection in such a way that these processes satisfy \eqref{eq:increment-domination}.

Finally, put
\begin{equation*}
X_n(t):=Y_1(t)+\dots+Y_n(t), \qquad \tilde{X}_n(t):=\tilde{Y}_1(t)+\dots+\tilde{Y}_n(t)
\end{equation*}
for $n\in\NN$, so that $\tilde{X}_n$ is the dual of $X_n$. Since the property \eqref{eq:increment-domination} is clearly preserved under addition of two or more processes, we have that, for all $0\leq s<t$,
\begin{equation}
X_n(t)-X_n(s)\geq\tilde{X}_n(t)-\tilde{X}_n(s)
\label{eq:partial-increment}
\end{equation}
pointwise on $\Omega$. Finally, note that $X_n(t)$ can be written as
\begin{align*}
X_n(t)&=(\gamma-L)t+\sigma B(t)+\int_{|y|\geq\eps_n}yN(t,dy)\\
&=\gamma_n t+\sigma B(t)+\int_{\eps_n\leq|y|<\eps}yN'(t,dy)+\int_{|y|\geq \eps}yN(t,dy),
\end{align*}
where
\begin{equation*}
\gamma_n:=\gamma-L+\int_{\eps_n\leq|y|<\eps}y\nu(dy).
\end{equation*}
By \eqref{eq:almost-symmetric}, $\gamma_n\to\gamma'$, and it follows from Theorem 2.6.2 in \cite{Applebaum} that $X_n(t)\to X(t)$ uniformly in $[0,T]$ with probability one, as long as the sequence $\{\eps_n\}$ decreases fast enough so that
\begin{equation}
\int_{0<|y|<\eps_n}y^2\nu(dy)\leq\frac{1}{8^n}
\label{eq:epsilon-condition}
\end{equation}
for every $n$. Similarly, $\tilde{X}_n(t)\to\tilde{X}(t)$ uniformly in $[0,T]$ with probability one. And, by taking limits in \eqref{eq:partial-increment}, we see that $X$ and $\tilde{X}$ satisfy \eqref{eq:increment-domination-general} everywhere on the set on which both processes converge. 
\end{proof}

\subsection{The case of bounded and continuous $f$} \label{subsec:bounded}

In general, it seems difficult to eliminate the unnatural condition (d) in the definition of SRSS, except when the reward function $f$ is bounded and continuous on $[0,\infty)$. This case includes, for instance, the natural reward function $f(x)=e^{-\sigma x}$ with $\sigma>0$.

Say a general L\'evy process $X=(X(t))_{t\geq 0}$ with L\'evy-Khintchine representation \eqref{eq:LK-representation-2} is {\em right skew symmetric} (RSS) if
\begin{equation}
\gamma\geq\liminf_{\delta\downarrow 0}\int_{\delta<|y|<1}y\nu(dy),
\label{eq:weaker-RSS-condition}
\end{equation}
and $\nu\big((a,\infty)\big)\geq\nu\big((-\infty,a)\big)$ for all $a>0$. Say $X$ is {\em left skew symmetric} (LSS) if $\tilde{X}$ is right skew symmetric.

\begin{theorem} \label{thm:bounded-f}
Let $X=(X(t))_{t\geq 0}$ be a L\'evy process, and let $f:[0,\infty)\to\RR$ be bounded, nonincreasing, continuous  and convex. For fixed $T>0$, consider the problem \eqref{eq:Levy-objective}. %\vspace{-2mm}
\begin{enumerate}[(i)]%\setlength{\itemsep}{-1mm}
\item If $X$ is RSS, the rule $\tau\equiv T$ is optimal.
\item If $X$ is LSS, the rule $\tau\equiv 0$ is optimal.
\end{enumerate}
\end{theorem}

(Observe that the symmetric case is already covered by Theorem \ref{thm:Levy-general}.)

\begin{proof}
Suppose first that $X$ is RSS. Let $L:=\liminf_{\delta\downarrow 0}\int_{\delta<|y|<1}y\nu(dy)$, and choose a sequence $\delta_1>\delta_2>\dots>0$ so that $\lim_{k\to\infty}\int_{\delta_k<|y|<1}y\nu(dy)=L$. For each $n$, choose $k_n$ so that $\eps_n:=\delta_{k_n}$ satisfies \eqref{eq:epsilon-condition}.
Now we construct the process $X$ as an almost-sure uniform limit of a sequence of processes $X_n=(X_n(t))_{t\geq 0}$, $n\in\NN$, exactly as in the proof of Lemma \ref{lem:general-construction}. Then each $X_n$ is RSS in the sense of Subsection \ref{subsec:interlacing}. (Note that in order to construct the processes $X_n$ in this way, without their duals, condition (d) in Definition \ref{def:SRSS} is not needed.) For each $t\geq 0$, let $\FF_t$ be the smallest $\sigma$-algebra containing each $\sigma(\{X_n(s): 0\leq s\leq t\})$, $n\in\NN$. Let $\Omega_0$ be the subset of $\Omega$ on which $X_n(t)$ converges uniformly in $t$. By arbitrarily redefining $X(t;\omega)\equiv 0$ for $\omega\in\Omega\backslash\Omega_0$, we see that $X$ is adapted to $\{\FF_t\}$, and clearly $X_t-X_s$ is independent of $\FF_s$ for each $0\leq s\leq t$. Thus, by Theorem \ref{thm:Levy-finite}, for any stopping time $\tau$ relative to $\{\FF_t\}$,
\begin{equation}
\rE\big[f\big(M_n(T)-X_n(\tau)\big)\big]\leq\sE\big[f\big(M_n(T)-X_n(T)\big)\big].
\label{eq:finite-n-inequality}
\end{equation}
Now it follows from the uniform convergence of $X_n$ to $X$ that, pointwise on $\Omega_0$, $M_n(T)\to M(T)$ and $X_n(\tau)\to X(\tau)$, and hence, by the continuity of $f$, $f\big(M_n(T)-X_n(\tau)\big)\to f\big(M(T)-X(\tau)\big)$ and $f\big(M_n(T)-X_n(T)\big)\to f\big(M(T)-X(T)\big)$. Thus, taking limits in \eqref{eq:finite-n-inequality} we see via the Bounded Convergence Theorem that
\begin{equation*}
\rE\big[f\big(M(T)-X(\tau)\big)\big]\leq \sE\big[f\big(M(T)-X(T)\big)\big].
\end{equation*}
Therefore, the rule $\tau\equiv T$ is optimal. A similar argument shows that the rule $\tau\equiv 0$ is optimal if $X$ is LSS. 
\end{proof}

\begin{remark}
{\rm
If we try to extend the above reasoning to unbounded continuous $f$ via the Dominated Convergence Theorem, we run into the difficulty of bounding expectations such as $\rE|f(M_n(T))|$ uniformly in $n$, since there is no guarantee that $\rE|f(M_n(T))|$ converges to $\rE|f(M(T))|$.
}
\end{remark}

\begin{remark}
{\rm
It may seem that in Theorem \ref{thm:Levy-general} we could have weakened the SRSS condition similarly, replacing (a) and (b) in Definition \ref{def:SRSS} with \eqref{eq:weaker-RSS-condition}. But this would not actually give a weaker hypothesis, since in the presence of condition (d), the integral in \eqref{eq:almost-symmetric} increases monotonically as $\eps\downarrow 0$.
}
\end{remark}

\section*{Acknowledgements}

This work was started while the author was on sabbatical in Kyoto, Japan. The author wishes to thank the Kyoto University Mathematics Department and the Research Institute for Mathematical Sciences for their warm hospitality during 2009. The author is grateful to two anonymous referees for several suggestions to improve the presentation of this paper.


\footnotesize
\begin{thebibliography}{14}

\bibitem{Allaart} \textsc{Allaart, P. C.} (2010). A general ``bang-bang" principle for predicting the maximum of a random walk. {\em J. Appl. Probab.} {\bf 47}, no. 4, 1072--1083.

\bibitem{Applebaum} \textsc{Applebaum, D.} (2009). {\em L\'evy processes and stochastic calculus}, 2nd edition, Cambridge University Press.

\bibitem{BDP1} \textsc{Bernyk, V., Dalang, R. C.} and \text{Peskir, G.} (2008). The law of the supremum of a stable L\'evy process with no negative jumps. {\em Ann. Probab.} {\bf 36}, 1777--1789. %\MR{2440923}

\bibitem{BDP2} \textsc{Bernyk, V., Dalang, R. C.} and \text{Peskir, G.} (2011). Predicting the ultimate supremum of a stable L\'evy process with no negative jumps. {\em Ann. Probab.} {\bf 39}, no. 6, 2385--2423.

\bibitem{Bertoin} \textsc{Bertoin, J.} (1996). {\em L\'evy processes}, Cambridge University Press.

\bibitem{CGMY} \textsc{Carr, P.}, \textsc{Geman, H.}, \textsc{Madan, D.} and \textsc{Yor, M.} (2002). The fine structure of asset returns: An empirical investigation. {\em J. Business} {\bf 75}, 305--332.

\bibitem{DuToit1} \textsc{Du Toit, J.} and \textsc{Peskir, G.} (2007). The trap of complacency in predicting the maximum. {\em Ann. Probab.} {\bf 35}, 340--365. %\MR{2303953}

\bibitem{DuToit2} \textsc{Du Toit, J.} and \textsc{Peskir, G.} (2009). Selling a stock at the ultimate maximum. {\em Ann. Appl. Probab.} {\bf 19}, 983--1014. %\MR{2537196} 

\bibitem{GPS} \textsc{Graversen, S. E.}, \textsc{Peskir, G.} and \textsc{Shiryaev, A. N.} (2000). Stopping Brownian motion without anticipation as close as possible to its ultimate maximum. {\em Theory Probab. Appl.} {\bf 45}, 41--50. %\MR{1810977}

\bibitem{Hlynka} \textsc{Hlynka, M.} and \textsc{Sheahan, J. N.} (1988). The secretary problem for a random walk. {\em Stoch. Proc. Appl.} {\bf 28}, 317--325. %\MR{952837}

\bibitem{Marshall} \textsc{Marshall, A. W.} and \textsc{Olkin, I.} (1979). {\em Inequalities: Theory of majorization and its applications}, Academic Press, New York.

\bibitem{Pedersen} \textsc{Pedersen, J. L.} (2003). Optimal prediction of the ultimate maximum of Brownian motion. {\em Stoch. Stoch. Rep.} {\bf 75}, 205--219. %\MR{1994906}

\bibitem{Sato} \textsc{Sato, K.-I.} (1999). {\em L\'evy processes and infinite divisibility}, Cambridge University Press.

\bibitem{SXZ} \textsc{Shiryaev, A. N.}, \textsc{Xu, Z.} and \textsc{Zhou, X. Y.} (2008). Thou shalt buy and hold. {\em Quant. Finance} {\bf 8}, 765--776. 
%\MR{2488734}

\bibitem{YYZ} \textsc{Yam, S. C. P.}, \textsc{Yung, S. P.} and \textsc{Zhou, W.} (2009). Two rationales behind `buy-and-hold or sell-at-once'. {\em J. Appl. Probab.} {\bf 46}, 651--668. %\MR{2560894}

\end{thebibliography}
\end{document}